\documentclass{amsart}
\newtheorem{thm}{Theorem}[section]
\newtheorem{cor}[thm]{Corollary}
\newtheorem{lem}[thm]{Lemma}
\newtheorem{prop}[thm]{Proposition}
\theoremstyle{definition}

\theoremstyle{remark}

\numberwithin{equation}{section}
\usepackage{graphicx}
\usepackage{amsmath,amsfonts,amssymb}
\usepackage{enumerate}
\begin{document}

\title[Discrete Asymptotic Nets with Constant Affine Mean Curvature]{Discrete Asymptotic Nets with Constant Affine Mean Curvature}%
\author{Anderson Reis de Vargas and Marcos Craizer}%
\email{}%

\thanks{Author's e-mail address: anderson.vargas.1@cp2.edu.br and craizer@puc-rio.br\\
Anderson R. Vargas is a teacher in Colégio Pedro II. Marcos Craizer is a professor in Pontifical Catholic University of Rio de Janeiro.}%

%\date{}%
%\dedicatory{}%
%\commby{}%
% ----------------------------------------------------------------
\begin{abstract}
In this paper we define the class of constant affine mean curvature (CAMC) discrete asymptotic nets, which contains the well-known classes of affine spheres and affine minimal asymptotic nets. This class is defined by considering fields of compatible interpolating quadrics, i.e., quadrics that have common tangent planes at the edges of the net. We show that, for CAMC asymptotic nets, ruled discrete asymptotic nets is equivalent to ruled compatible interpolating quadrics. Moreover, 
we prove discrete counterparts of some known properties of the Demoulin transform of a smooth CAMC surface. 
\end{abstract}
\maketitle
% ----------------------------------------------------------------

\section{Introduction}

In this paper we consider discrete asymptotic nets, which are natural nets for the discretization
of surfaces parameterized by asymptotic coordinates. Discrete asymptotic nets have been the object of
recent and ancient research by many geometers, as one can see in the list of references of
this paper (\cite{Bobenko1999},\cite{Bobenko2008},\cite{Craizer2010},\cite{Rorig2014},\cite{Kaferbock2013},\cite{Matsuura2003},\cite{McCarthy-Schief},\cite{Schief-Szeres}).
There are many classes of discrete affine surfaces that have been defined as subclasses of discrete asymptotic nets:
Discrete affine spheres (\cite{Bobenko1999}), discrete improper
affine spheres (\cite{Matsuura2003}), and discrete affine minimal surfaces (\cite{Craizer2010},\cite{Kaferbock2013}). In this paper we define the constant affine mean curvature (CAMC) discrete asymptotic nets, which is a class that contains all the above classes. 

A quadric patch interpolates a quadrangle of a discrete asymptotic net if it passes through the edges of the quadrangle and is contained in its convex hull. A pair of interpolating quadrics at adjacent quadrangles will be called {\it compatible} if, at the common edge of the quadrangles, the tangent planes of the quadrics coincide. A field of interpolating quadrics, one at each quadrangle of the asymptotic net, will be called {\it compatible} if they are compatible at each pair of adjacent quadrangles. 
It is well-known that, given a discrete asymptotic net, there exists a $1$-parameter family of compatible fields of interpolating quadrics (\cite{Rorig2014}).  Moreover, a discrete asymptotic net is a discrete affine sphere if and only if it admits a compatible field of interpolating quadrics with the same center. Similarly, a discrete asymptotic net is affine minimal if and only if it admits a compatible field of interpolating paraboloids. We define the CAMC class as the discrete asymptotic nets that admit a compatible field of interpolating quadrics with the same affine mean curvature. %It is clear that the CAMC class contains the affine sphere and affine minimal classes. 

We prove some discrete counterparts of known results for smooth CAMC surfaces that reinforce the above definition of discrete CAMC asymptotic nets. The first one is related to ruled discrete asymptotic nets. We show that, in the CAMC class, a ruled discrete asymptotic net admits a compatible field of interpolating quadrics that is constant along the ruling direction. 

The other results that we prove for discrete CAMC asymptotic nets are related to Demoulin transforms. We first prove a discrete counterpart of the result that says that, in the smooth case, the intersection of the affine normal with the Lie quadric determines a surface that is a Demoulin transform of the original surface if and only if the affine mean curvature is constant. Then we prove that, as occur in the smooth case, a discrete affine minimal asymptotic net is a discrete $Q$-surface.  Finally, we prove that an affine sphere is Demoulin, and conversely, a Demoulin CAMC asymptotic net is an affine sphere, results which also hold in the smooth category.

\section{Affine Properties of Smooth Surfaces}

In this section we recall some affine properties of a non-degenerate smooth surface in $3$-space with indefinite Blaschke metric.

\subsection{Affine equations in asymptotic coordinates}

A non-degenerate smooth surface with indefinite Blaschke metric admits asymptotic coordinates. Consider an immersion
$f:U\subset\mathbb{R}^2\to\mathbb{R}^3$, such that $(u,v)\in U$ are asymptotic coordinates. Then we can write
\begin{equation*}
f_{uu}=\frac{\omega_u}{\omega}f_u+\frac{A}{\omega}f_v,\ \ f_{vv}=\frac{B}{\omega}f_u+\frac{\omega_v}{\omega}f_v,
\end{equation*}
where $\omega^2=\left[f_u,f_v,f_{uv}\right]$, and $Adu^3+Bdv^3$ is the cubic form (see e.g. \cite[ch.1]{Buchin1983}). Moreover, denoting 
by $\xi$ the affine Blaschke normal vector field, we have that
\begin{equation*}
f_{uv}=\omega \xi,
\end{equation*}
and
\begin{equation}\label{eq:Shape}
\xi_u=-Hf_u+\frac{A_v}{\omega^2}f_v, \ \ \xi_v=\frac{B_u}{\omega^2}f_u-Hf_v,
\end{equation}
where $H$ is the affine mean curvature. The compatibility equations are
\begin{equation}\label{eq:Compatibility1}
\omega^3H= \omega_u\omega_v-\omega_{uv}\omega  -AB
\end{equation}
together with
\begin{equation*}
\omega^3H_u=AB_u-\omega A_{vv}+A_v\omega_v,\ \ {\rm and} \ \ \omega^3H_v=BA_v-\omega B_{uu}+B_u\omega_u.
\end{equation*}

%\begin{equation*}
%\omega^3H_v=BA_v-\omega B_{uu}+B_u\omega_u.
%\end{equation*}

%One can also verify that
%\begin{equation*}
%f_{uuv}=-\omega Hf_u+\frac{A_v}{\omega}f_v+\frac{\omega_u}{\omega}f_{uv},
%\end{equation*}
%\begin{equation*}
%f_{uvv}=\frac{B_u}{\omega}f_u-\omega Hf_v+\frac{\omega_v}{\omega}f_{uv},
%\end{equation*}

\subsection{Lie quadric}

The Lie quadric at a point $(u,v)$ of the surface $f$ is the limit of the quadric that passes through three infinitesimally close asymptotic lines. In homogeneous coordinates, we write $\bar{f}=(f:1)$, and the Lie quadric at a point $(u,v)$ is given by
\begin{equation}\label{eq:LieQuadric1}
Q(s,t)=\bar{f}_{uv}+\left( s-\frac{A_v}{2A} \right)\bar{f}_u+\left( t-\frac{B_u}{2B}  \right)\bar{f}_v+\left( \frac{\omega H}{2}+\frac{A_vB_u}{4AB} -t\frac{A_v}{2A}-s\frac{B_u}{2B}+st  \right)\bar{f},
\end{equation}
(\cite{Ferapontov},\cite{Sasaki}). Writing $Q=X_1f+X_2f_u+X_3f_v+X_4f_{uv}$, we obtain
\begin{equation}\label{eq:LieQuadric2}
X_2X_3-X_1X_4+\frac{1}{2}\omega H X_4^2=0,
\end{equation}
which, by Equation \eqref{eq:Compatibility1}, coincides with  \cite[p.145]{Lane1942}. By taking $X_1=1$ in Equation \eqref{eq:LieQuadric2}, we obtain an affine equation of the Lie quadric. From this affine equation, it is easy to verify that the affine mean curvature of the Lie quadric
at a point $(u,v)$ coincides with the affine mean curvature of the original surface at the same point.

\subsection{Demoulin transforms of a CAMC surface}

The affine normal line of $f$ at a point $(u,v)$ is given by
$\lambda\to f+\lambda f_{uv}.$
Taking $s=\frac{A_v}{2A}$, $t=\frac{B_u}{2B}$ in Equation \eqref{eq:LieQuadric1}, we obtain that, besides $f(u,v)$, the intersection of the normal line with the Lie quadric is at the point
\begin{equation}\label{eq:DemoulinTransform}
Z=\bar{f}_{uv}+\left( \frac{\omega H}{2}  \right)\bar{f}.
\end{equation}

\begin{prop}\label{prop:DemoulinCAMC}
The surface $Z$ given by Equation \eqref{eq:DemoulinTransform} is tangent to the Lie quadric if and only if $H_u=H_v=0$. 
\end{prop}

\begin{proof}
Observe that, at $s=\frac{A_v}{2A}$, $t=\frac{B_u}{2B}$,
$$
Q=\bar{f}_{uv}+\left( \frac{\omega H}{2}  \right)\bar{f}, \ \ Q_s=\bar{f}_u,\ \ Q_t=\bar{f}_v. 
$$
On the other hand
$$
Z_u=\frac{\omega_u}{\omega}\bar{f}_{uv}-\frac{\omega H}{2}\bar{f}_u+\frac{A_v}{\omega}\bar{f}_v+\frac{1}{2}\left( \omega_uH+\omega H_u\right)\bar{f}.
$$
Thus, denoting $\Delta=\left[\bar{f},\bar{f}_u,\bar{f}_v,\bar{f}_{uv}\right]$, we have 
$$
[Q,Q_s,Q_t,Z_u]=\left(\frac{\omega_uH}{2}-\frac{1}{2}(\omega_uH+\omega H_u)\right) \Delta=-\frac{\omega H_u}{2}\Delta.
$$
Similarly
$$
[Q,Q_s,Q_t,Z_v]=-\frac{\omega H_v}{2}\Delta,
$$
thus proving the proposition.
\end{proof}

From Proposition \ref{prop:DemoulinCAMC}, we conclude that $Z$ is a Demoulin transform of $f$  if and only if $f$ is CAMC. For a definition of the Demoulin transform, see (\cite{Ferapontov-Schief},\cite{Ferapontov},\cite{Sasaki}).

\subsection{Affine spheres}

In case $f$ is an affine sphere, then the four Demoulin transforms coincide with $Z$ given by Equation \eqref{eq:DemoulinTransform}, which says that $f$ is a Demoulin surface (\cite{Ferapontov-Schief},\cite{Ferapontov},\cite{Sasaki}).

Conversely, if $f$ is CAMC, then straightforward calculations show that the Demoulin transforms of $f$ are obtained by taking 
$$
s=\pm \frac{A_v}{2A}, \ \ t=\pm \frac{B_u}{2B}
$$
in Equation \eqref{eq:LieQuadric1}. Thus $f$ is Demoulin if and only if $A_v=B_u=0$, and these conditions imply that $f$ is an affine sphere. 

\subsection{Affine minimal surfaces}

In case $f$ is an affine minimal surface, then Equation \eqref{eq:DemoulinTransform} reduces to 
$$
Z=\bar{f}_{uv}=(f_{uv}:0)=(\xi:0). 
$$
Thus we may take the projective version of the affine normal vector $\xi$ as a representative of $Z$. Differentiating Equation \eqref{eq:Shape} with $H=0$ gives
$$
\xi_{uu}=\left(\frac{A_{uv}}{A_v}-\frac{2\omega_u}{\omega}\right) \xi_u + \frac{A_v}{\omega} \xi,
$$
which implies that the projective version of $\xi(\cdot,v)$ is a straight line. Similarly, the projective version of $\xi(u,\cdot)$
is also a straight line. Thus $Z$ is a Demoulin transform of $f$ whose $u$-curves and $v$-curves are straight lines. We conclude
that $f$ is a $Q$-surface. For a definition of $Q$-surface, see (\cite{McCarthy-Schief},\cite{Schief-Szeres}).

%\begin{rem}One can also prove algebraically that an affine minimal surface is a $Q$-surface by considering the condition given 
%in \cite[sec.2(b)]{Schief2018}. But we shall not do it here, since it would take too much space to explain the notations. 
%\end{rem}

\section{Co-normal Map and Interpolations on Discrete Asymptotic Nets}

\subsection{Basic definitions and notations}

A discrete asymptotic net is a map $q:\mathcal{D}\subset\mathbb{Z}^2\to\mathbb{R}^3$ whose ``crosses are planar'', i.e., for each $(i,j)\in\mathcal{D}$, the five points $q(i,j)$, $q(i\pm 1,j)$ and $q(i,j\pm 1)$ are co-planar. Denote by $D^*$ the set of squares $(i+\tfrac{1}{2},j+\tfrac{1}{2})$ whose vertices $q(i,j)$, $q(i+1,j)$, $q(i,j+1)$ and $q(i+1,j+1)$ belong to $\mathcal{D}$. We shall assume that any pair of squares in $\mathcal{D}^*$ can be connected by a chain of adjacent squares, where by adjacent we mean that they have a common side. 

For each function $f:\mathcal{D}\to\mathbb{R}^k$, we shall denote 
$$
f_1(i+\tfrac{1}{2},j)=f(i+1,j)-f(i,j),\ \ \ f_2(i,j+\tfrac{1}{2})=f(i,j+1)-f(i,j).
$$
Similarly, for $f:\mathcal{D}^*\to\mathbb{R}^k$,
$$
f_1(i,j+\tfrac{1}{2})=f(i+\tfrac{1}{2},j+\tfrac{1}{2})-f(i-\tfrac{1}{2},j+\tfrac{1}{2}),\ \ \ f_2(i+\tfrac{1}{2},j)=f(i+\tfrac{1}{2},j+\tfrac{1}{2})-f(i+\tfrac{1}{2},j-\tfrac{1}{2}).
$$

For each square $(i+\tfrac{1}{2},j+\tfrac{1}{2})$, denote 
\begin{equation*}
\delta(i+\tfrac{1}{2},j+\tfrac{1}{2})=\left[ q_1(i+\tfrac{1}{2},j), q_2(i,j+\tfrac{1}{2}), q_{12}(i+\tfrac{1}{2},j+\tfrac{1}{2})\right],
\end{equation*}
where $[\cdot,\cdot,\cdot]$ denotes the determinant of three vectors in $\mathbb{R}^3$.  We say that the asymptotic net is non-degenerate if $\delta$ does not change sign. Along this paper, we shall be assuming that the asymptotic net is non-degenerate, and, without loss of generality, we may assume $\delta>0$. 
We shall then write
\begin{equation}\label{eq:DefineOmega}
\Omega^2(i+\tfrac{1}{2},j+\tfrac{1}{2})=\left[ q_1(i+\tfrac{1}{2},j), q_2(i,j+\tfrac{1}{2}), q_{12}(i+\tfrac{1}{2},j+\tfrac{1}{2})\right],
\end{equation}
and call $\Omega$ the affine metric of the discrete asymptotic net. For examples of singular discrete asymptotic nets, i.e., asymptotic nets where $\delta$ changes sign, see \cite{Vargas2022}.

\subsection{Co-normal map}

The co-normal vector field $\nu$ with respect
to a discrete asymptotic net $q$ is a vector-valued map defined at the vertices $(i, j)\in\mathcal{D}$
satisfying the discrete Lelieuvre’s equations
\begin{equation}\label{eq:Lelieuvre}
\nu(i,j)\times\nu(i+1,j)=q_1(i+\tfrac{1}{2},j),\ \ \nu(i,j)\times\nu(i,j+1)=-q_2(i,j+\tfrac{1}{2}).
\end{equation}
%We shall assume $[\nu(i,j),\nu(i+1,j),\nu(i,j+1)]>0$. 

For a constant $\rho$, consider
\begin{equation*}\label{eq:BWRescaling}
\nu_{\rho}(i,j)=\rho\nu(i,j),\ \ i+j\ \ {\rm even}, \ \ \nu_{\rho}(i,j)=\rho^{-1}\nu(i,j),\ \ i+j \ \ {\rm odd}.
\end{equation*}
This operation is known as black-white re-scaling and it is easy to show that $\nu_{\rho}$ also satisfies Lelieuvre’s
equations, i.e., is a way to obtain another conormal vector field with
respect to the same $q$ net. Conversely, we can see that any conormal vector field
with respect to the asymptotic net $q$ is obtained from $\nu$ by a black-white
re-scaling. In particular $\nu$ is completely determined by its value at one vertex.

\begin{prop} A vector field $\nu(i,j)$ is the conormal vector field of an
asymptotic net $q(i, j)$ if and only if it defines a Moutard net, i.e.,
\begin{equation}\label{eq:Moutard}
\lambda^2(i+\tfrac{1}{2},j+\tfrac{1}{2})(\nu(i,j)+\nu(i+1,j+1))=\nu(i,j+1)+\nu(i+1,j),
\end{equation}
for some map $\lambda:\mathcal{D}^*\to\mathbb{R}_{+}^{*}$. Moreover, 
\begin{equation}\label{eq:OmegaLambdaNu}
\Omega(i+\tfrac{1}{2},j+\tfrac{1}{2})=\frac{1}{\lambda(i+\tfrac{1}{2},j+\tfrac{1}{2})} [\nu(i,j),\nu(i+1,j),\nu(i,j+1)].
\end{equation}
\end{prop}

\begin{proof}
Observe that
$$
(q_2)_1(i+\tfrac{1}{2},j+\tfrac{1}{2})=-\nu(i+1,j)\times\nu(i+1,j+1)+\nu(i,j)\times\nu(i,j+1),
$$
$$
(q_1)_2(i+\tfrac{1}{2},j+\tfrac{1}{2})=\nu(i,j+1)\times\nu(i+1,j+1)-\nu(i,j)\times\nu(i+1,j),
$$
and so $(q_1)_2=(q_2)_1$ if and only if 
$$
(\nu(i+1,j+1)+\nu(i,j))\times (\nu(i+1,j)+\nu(i,j+1))=0, 
$$
thus proving that 
$$
\alpha(i+\tfrac{1}{2},j+\tfrac{1}{2})(\nu(i,j)+\nu(i+1,j+1))=\nu(i,j+1)+\nu(i+1,j), 
$$
for some $\alpha:(\mathbb{Z}^2)^*\to\mathbb{R}\setminus\{0\}$. Now observe that, from Equations \eqref{eq:DefineOmega} and \eqref{eq:Lelieuvre}, we obtain
$$
\Omega^2=[\nu(i,j),\nu(i+1,j),\nu(i,j+1)][\nu(i,j),\nu(i+1,j),\nu(i+1,j+1)]
$$
and so
$$
\Omega^2=\frac{1}{\alpha} [\nu(i,j),\nu(i+1,j),\nu(i,j+1)]^2
$$
We conclude that $\alpha>0$ and so we can write $\alpha=\lambda^2$ with $\lambda>0$, thus proving the proposition.
\end{proof}

%\begin{lem}
%We have that
%$$
%\nu_1(i+\tfrac{1}{2},j)\cdot q_1(i+\tfrac{1}{2},j)=0, \ \ \nu_2(i,j+\tfrac{1}{2})\cdot q_2(i,j+\tfrac{1}{2})=0,
%$$
%and 
%$$
%\nu_1(i+\tfrac{1}{2},j)\cdot q_2(i,j+\tfrac{1}{2})=\nu_2(i,j+\tfrac{1}{2})\cdot q_1(i+\tfrac{1}{2},j)=\lambda(i+\tfrac{1}{2},j+\tfrac{1}{2})\Omega(i+\tfrac{1}{2},j+\tfrac{1}{2}).
%$$
%\end{lem}

%\begin{proof}
%We have that 
%$$
%-(\nu(i+1,j)-\nu(i,j))\cdot \nu(i,j)\times\nu(i,j+1)=[\nu(i,j),\nu(i+1,j),\nu(i,j+1)], 
%$$
%thus proving the lemma.
%\end{proof}

%\section{Interpolation by Quadrics}

\subsection{Interpolations at a single quadrangle}

In order to get smaller formulas, given a quadrangle $(i+\tfrac{1}{2},j+\tfrac{1}{2})\in \mathcal{D}^*$, we shall use the notation 
$$
A=q(i,j),\ \ B=q(i+1,j), \ \ C=q(i,j+1), \ \ D=q(i+1,j+1).
$$
Given four non coplanar points $A$, $B$, $C$ and $D$ in $\mathbb{R}^3$, we are looking for a quadric patch that contains the segments $AB$, $CD$, $AC$ and $BD$ and its convex hull coincides with the tetrahedron $ABCD$.

\begin{lem} 
There exists a $1$-parameter family of ruled quadrics passing through the edges of $ABCD$, with patches in the convex hull of the quadrangle. For the standard quadrangle $A=(0, 0, 0)$, $B=(1, 0, 0)$, $C=(0, 1, 0)$, $D=(1, 1, 1)$, the cartesian equation of this family is 
\begin{equation}\label{eq:BasicCartesianHyperboloid}
z(1+z-x-y)=(a+1)(z-y)(z-x), \ \ a>-1.
\end{equation}
\end{lem}

\begin{proof}
Consider a general quadric given by equation
$$
Ax^2+By^2+Cz^2+Dxy+Exz+Fyz+Gx+Hy+Iz=J
$$
If we impose the conditions that it contains the edges of $ADCD$ we conclude that
$$
A=B=G=H=J=0, \ E=F=a,\ C=-a, \ D=-I-a.
$$
If $I=0$, then we obtain the pair of planes $(z-y)(z-x)=0$. If $I\neq 0$, we may assume $I=1$. Then we obtain the family 
\eqref{eq:BasicCartesianHyperboloid}. 
If $a=-1$, then we obtain the pair of planes $z(1+z-x-y)=0$. The convex hull of $ABCD$ is defined by the equations
$$
z\geq 0,\ \ 1+z-x-y\geq 0, \ \ z-y\leq 0, \ \ z-x\leq 0. 
$$
Thus for the quadric \eqref{eq:BasicCartesianHyperboloid} to contain a patch in the convex hull of $ABCD$ we must have $a>-1$, thus proving the lemma.  
\end{proof}

Considering the standard quadrangle $A=(0, 0, 0)$, $B=(1, 0, 0)$, $C=(0, 1, 0)$, $D=(1, 1, 1)$, given $a>-1$, define the interpolator 
\begin{equation}\label{StandardInterpolator}
\phi(u,v)=\frac{1}{1+auv}\left( u+auv, v+auv, (1+a)uv \right),\ \ 0\leq u,v\leq 1,
\end{equation}
(Figure \ref{Fig:StandardInterpolator}). Let
\begin{equation*}\label{eq:EtaNuStandard}
\nu(A)=\eta(0,0), \nu(B)=\eta(1,0), \nu(C)=\eta(0,1), \nu(D)=\eta(1,1),
\end{equation*}
where $\eta$ denote the affine co-normal of the interpolator \eqref{StandardInterpolator}.

\begin{figure}[!htb]
 \includegraphics[width=.48\linewidth]{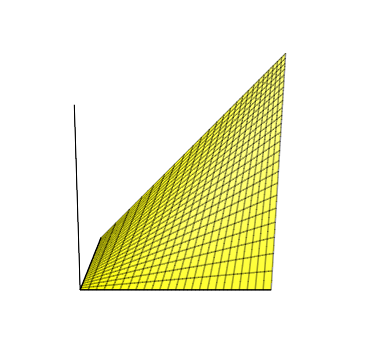}
% \hfill
% \includegraphics[width=.48\linewidth]{Ex-sing-vistasuperior.eps}
\caption{\small Standard interpolating quadric defined by Equation \eqref{StandardInterpolator} with $a=1$.}\label{Fig:StandardInterpolator}
\end{figure}

\begin{lem}\label{lemma:BasicAffineQuadric}
Consider the interpolator given by Equation \eqref{StandardInterpolator}. The center of the quadric is the point $O=\tfrac{1}{2}(1, 1,\tfrac{a+1}{a})$ and its affine mean curvature is $H=-\frac{2a}{\sqrt{a+1}}$.
Moreover, the co-normal $\nu$ satisfies the discrete Lelieuvre's equations \eqref{eq:Lelieuvre} in the standard quadrangle $ABCD$. It also satisfies the discrete Moutard equation \eqref{eq:Moutard} with 
\begin{equation*}\label{eq:LambdaA}
\lambda^2=a+1.
\end{equation*}
\end{lem}

\begin{proof}
Straightforward calculations show that $L=N=0$, which means that $(u,v)$ are asymptotic
coordinates, and 
$$
M =\frac{a+1}{(1+auv)^{4}}, \ \ \omega(u, v) = \frac{\sqrt{a+1}}{\left(1+auv\right)^2}.
$$
Since
$$
\xi(u,v)=-\frac{2a}{\sqrt{a+1}}\left(   \phi-\frac{1}{2}\left(1,1,\frac{a+1}{a}\right) \right),
$$
the quadric has center at the
point $O$ and affine mean curvature $H$. Straightforward calculations show that 
the affine co-normal $\eta$ is given by
$$
\eta(u,v)=\frac{1}{(1+auv)\sqrt{a+1}}(-(a+1)v,-(a+1)u,1+au+av-auv),
$$
which implies that
$$
\nu(A)=\tfrac{1}{\sqrt{a+1}}\left(0,0,1\right), \nu(B)=\sqrt{a+1}\left( 0, -1, 1  \right), 
$$
$$
\nu(C)=\sqrt{a+1}\left( -1, 0, 1  \right), \nu(D)=\tfrac{1}{\sqrt{a+1}}\left(-1,-1,1\right).
$$
It is now easy to verify the discrete Lelieuvre's and Moutard's formulas in the standard quadrangle.
\end{proof}

For a generic quadrangle $ABCD$, consider the affine map $T$ composed by a translation of $A$ with the
linear map that takes 
$$
(1, 0, 0)\to B-A, \ (0, 1, 0)\to C-A, \ (0, 0, 1)\to D+A-B-C.
$$
Then a parametric equation in asymptotic coordinates of the patch is
\begin{equation}\label{GeneralInterpolator}
T\circ\phi(u,v)=A+\frac{1}{1+auv}\left(  u(B-A)+v(C-A)+uv((1+a)D+(1-a)A-B-C)\right),
\end{equation}
where $a>-1$. Since $\det(T)=\Omega^2$, the affine mean curvature of the interpolator \eqref{GeneralInterpolator} is given by 
\begin{equation*}\label{eq:AMC}
H=-\frac{2a}{\sqrt{1+a}\ \Omega}.
\end{equation*}
Moreover, also in the case of a general quadrangle, the discrete co-normal $\nu$ satisfies the discrete Lelieuvre's equations \eqref{eq:Lelieuvre} and the discrete Moutard equation \eqref{eq:Moutard} with $\lambda^2=a+1$. 

%One can also verify that
%\begin{equation}\label{eq:OmegaNu}
%\lambda\Omega=[\nu(A),\nu(B),\nu(C)].
%\end{equation}

It is useful to understand what happens with the parameter $a$ if we change the roles of $A,B,C$ and $D$. If we want to obtain the same 
interpolator quadric and the same orientation, we can consider only the ordered permutations 
$(B,D,A,C)$, $(C,A,D,B)$ and $(D,C,B,A)$. 

\begin{lem}
In order to obtain the same quadric interpolator, the parameter $a$ of Equation \eqref{GeneralInterpolator} must be the same for the permutation $(D,C,B,A)$. For the permutations $(B,D,A,C)$ and $(C,A,D,B)$, the correct parameter is 
\begin{equation*}
\bar{a}=-\frac{a}{a+1}.
\end{equation*}
\end{lem}

\begin{proof}
We may assume that $ABCD$ is the standard quadrangle. Thus $\phi$ is given by Equation \eqref{StandardInterpolator}. We shall consider the permutation $(B,D,A,C)$, the other cases being analogous. For this permutation, Equation \eqref{GeneralInterpolator} gives
$$
\psi(u,v)=T\circ\phi(u,v)=\frac{1}{1+\bar{a}uv}\left(1-v, u+\bar{a}uv,u-uv \right).
$$
Consider the projective change of coordinates 
$$
v=1-t,\ \ u=\frac{s(1+a)}{1+as}. 
$$
Then 
$$
\psi(s,t)=\frac{1}{1+ast}\left( t(1+as), s(1+at),(1+a)st\right),
$$
thus proving the lemma.
\end{proof}

To conclude this section, we observe that the parameter $\lambda=\sqrt{a+1}$ remains the same for the permutation  
$(D,C,B,A)$, but changes to $\bar{\lambda}=\lambda^{-1}$ for the permutations $(B,D,A,C)$ and $(C,A,D,B)$. Similarly, the affine mean curvature remains the same for the permutation  
$(D,C,B,A)$, but changes to $\bar{H}=-H$ for the permutations $(B,D,A,C)$ and $(C,A,D,B)$. In fact,
$$
\bar{H}=-\frac{2\bar{a}}{\sqrt{\bar{a}+1}\Omega}=\frac{2a\sqrt{1+a}}{(a+1)\ \Omega}=-H.
$$

\subsection{Compatible interpolations at a pair of adjacent quadrangles}

We say that a pair of interpolating quadrics at adjacent quadrangles are {\it compatible} if the tangent plane of both quadrics coincide at the common edge. 

Consider a pair of adjacent quadrangles $ABCD$ and $ACEF$. We may assume that
\begin{equation}\label{StandardPair}
A=(0,0,0), B=(1,0,0), C=(0,1,0), D=(1,1,1), E=(x_1,y_1,0), F=(x_2,y_2,x_2),
\end{equation}
for some real numbers $x_1,y_1,x_2,y_2$. $E$ and $F$ were chosen
in a way that the crosses are planar at $A$ and $C$. Moreover, the non-degeneracy condition implies that both 
$x_1$ and $x_2$ are strictly negative (see Figure \ref{Fig:Rede1}). 

\begin{figure}[!htb]
 \includegraphics[width=.48\linewidth]{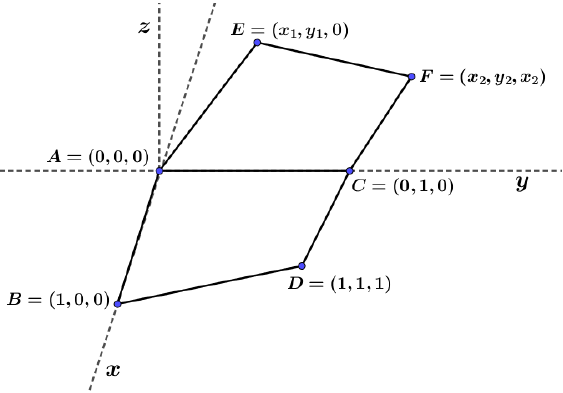}
% \hfill
% \includegraphics[width=.48\linewidth]{Ex-sing-vistasuperior.eps}
\caption{\small Pair of adjacent quadrangles defined by vertices \eqref{StandardPair}.}\label{Fig:Rede1}
\end{figure}

A general interpolating quadric for $ABCD$ is given by Equation \eqref{StandardInterpolator}. Denote by $b$ the parameter of the quadric interpolator at the quadrangle $EAFC$. In order to simplify the calculations, we shall consider the non-standard permutation $ACEF$, and so the corresponding parameter is $\bar{b}$. 
With this permutation, Equation \eqref{GeneralInterpolator} becomes
\begin{equation}\label{eq:Psi}
\psi(s,t)=\frac{1}{1+\bar{b}st}\left( tx_1+st((1+\bar{b})x_2-x_1) , ty_1 +s+st((1+\bar{b})y_2-y_1-1), st (1+\bar{b})x_2 \right).
\end{equation}

\begin{figure}[!htb]
 \includegraphics[width=.48\linewidth]{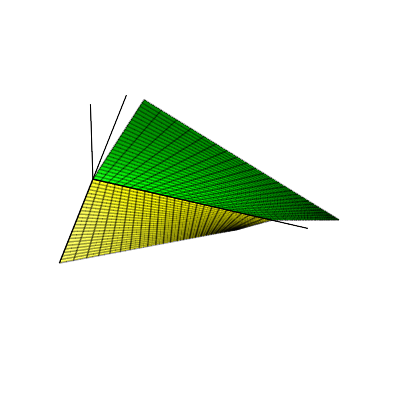}
% \hfill
% \includegraphics[width=.48\linewidth]{Ex-sing-vistasuperior.eps}
\caption{\small Compatible quadric interpolators at adjacent quadrangles defined by Equations \eqref{StandardInterpolator} and \eqref{eq:Psi} with $a=1$, $x_1=-1$, $x_2=-0.9$, $y_1=0.1$, $y_2=1.2$, $\bar{b}=1.22$.}\label{Fig:StandardPair}
\end{figure}

Along the paper, we shall often use this pair of quadrangles $ABCDEF$ given by \eqref{StandardPair} together with interpolating quadrics $\phi$ and $\psi$ given by Equations \eqref{StandardInterpolator} and \eqref{eq:Psi} for making calculations (Figure \ref{Fig:StandardPair}).

\begin{lem}\label{lemma:ConditionCompatibility}
The interpolating quadrics $\phi$ and $\psi$ given by Equations \eqref{StandardInterpolator} and \eqref{eq:Psi} of the adjacent quadrangles $ABCDEF$ given by \eqref{StandardPair} are compatible along the common edge $AC$ if and only if
\begin{equation}
(1+a)(1+b)x_1=x_2.
\end{equation}
\end{lem}

\begin{proof}
In order to make calculations simpler, consider homogeneous coordinates 
$(X_1,X_2,X_3,X_4)$. Then
$$
\phi(u,v)=\left( u+auv, v+auv, (1+a)uv, 1+auv \right).
$$
At a point of the segment $AC$, we have
\begin{equation}\label{eq:PhiU}
\phi_v(0,v)=(0,1,0,0),\ \ \phi_u(0,v)=\left( 1+av, av, (1+a)v, av \right).
\end{equation}
Thus the equation of the tangent plane is
\begin{equation}\label{eq:TangentPlaneEdge}
(1+a)vX_1=(1+av)X_3.
\end{equation}
In homogeneous coordinates, we can write Equation \eqref{eq:Psi} as
$$
\psi(s,t)=\left( tx_1+st((1+\bar{b})x_2-x_1) , ty_1 +s+st((1+\bar{b})y_2-y_1-1), st (1+\bar{b})x_2, 1+\bar{b}st \right).
$$
At the segment $AC$, $t=0$ and $\psi(s,0)=(0,s,0,1)$. Thus $s=v$. Moreover
\begin{equation}\label{eq:PsiT}
\psi_t(s,0)=\left( x_1+s((1+\bar{b})x_2-x_1) , y_1 +s((1+\bar{b})y_2-y_1-1), s (1+\bar{b})x_2, \bar{b}s \right).
\end{equation}
Then $\psi_t(s,0)$ belongs to the tangent plane \eqref{eq:TangentPlaneEdge} if and only if
$$
(1+a) (x_1+s((1+\bar{b})x_2-x_1))= (1+as) (1+\bar{b})x_2.
$$
which is equivalent to $(1+a)x_1=(1+\bar{b})x_2$, thus proving the lemma.
\end{proof}

It is clear from Lemma \ref{lemma:ConditionCompatibility} that, under compatibility, the choice of an interpolator quadric at a quadrangle determines the choice of the interpolator quadrics at all other quadrangles of $\mathcal{D}^*$. 

\begin{prop}\label{prop:ConormalCompatibility}
Consider a pair of adjacent quadrangles with compatible interpolations. Then the co-normal of both interpolators at a common 
vertex coincide.
\end{prop}

\begin{proof}
We may assume that the adjacent quadrangles are given by \eqref{StandardPair} with interpolating quadrics \eqref{StandardInterpolator} and \eqref{eq:Psi} and that the common vertex is $A$. Then
$$
\eta_{\phi}(A)=\frac{1}{\sqrt{1+a}}(0,0,1),\ \ {\rm and}\ \ \eta_{\psi}(A)=\frac{\sqrt{-x_1}}{\sqrt{-x_2}\sqrt{1+\bar{b}}}(0,0,1).
$$
Thus $\eta_{\phi}(A)=\eta_{\psi}(A)$ if and only if $(1+\bar{b})x_2=(1+a)x_1$, which, by Lemma \ref{lemma:ConditionCompatibility}, is exactly the compatibility condition.
\end{proof}

\subsection{Field of compatible interpolating quadrics}

A field of interpolating quadrics is a map from $\mathcal{D}^*$ to the set of interpolating quadrics. 
We say that a field of interpolating quadrics is compatible if, for each adjacent pair of quadrangles, the corresponding quadrics are compatible. In \cite{McCarthy-Schief} and \cite{Schief-Szeres}, an interpolating quadric is called a {\it discrete Lie quadric} and a compatible field of interpolating quadrics is called a {\it lattice of Lie quadrics}.

There is a one parameter choice for the co-normal vector at a fixed vertex $(i_0,j_0)$ of the asymptotic net. Then Lelieuvre's equation define the co-normal vector at all other vertices of the net. Moreover, Moutard's equation guarantees that the definition of the co-normal vector at a vertex $(i,j)$ is independent of the choice of the path from $(i_0,j_0)$ to $(i,j)$. From Proposition
\ref{prop:ConormalCompatibility}, the choice of the co-normal vector field determine a compatible field of interpolators along the quadrangles of the net. We have thus given another proof of a well-known result which says that any discrete asymptotic net admits a $1$-parameter family of fields of compatible interpolating quadrics (\cite{Rorig2014}). 

\section{Asymptotic Nets with Constant Affine Mean Curvature}

We say that a discrete asymptotic net admits a {\it constant affine mean curvature} (CAMC) structure if, among the $1$-parameter family
of quadric interpolator fields, there exists one whose affine mean curvature at each quadrangle is constant. 
%In other words, there exists an interpolator $\mathcal{H}$ consisiting of compatible quadrics
%$\mathcal{H}(i+\tfrac{1}{2},j+\tfrac{1}{2})$, $(i+\tfrac{1}{2},j+\tfrac{1}{2})\in \mathcal{D}^*$, such that the affine mean curvature of $\mathcal{H}(i+\tfrac{1}{2},j+\tfrac{1}{2})$ does not depend on $(i,j)$.

\subsection{Basic properties}

Consider a pair of adjacent squares together with compatible intepolating quadrics with parameters $a$ and $b$. Denote by $\Omega_i$ the discrete affine metric at the quadrangle $i$. 
Then the condition for CAMC can be written as
\begin{equation}\label{eq:CAMC}
\frac{a}{\sqrt{1+a}\Omega_1}=\frac{b}{\sqrt{1+b}\Omega_2}.
\end{equation}

\begin{lem}\label{lemma:CAMCAdjacentQuadrangles}
Consider a discrete asymptotic net formed by a pair of adjacent quadrangles. Then there exists exactly a pair $(a,b)$ such that the asymptotic net becomes CAMC. For the pair of quadrangles $ABCDEF$ given by \eqref{StandardPair}, the condition for CAMC is $ax_1=\bar{b}$. 
\end{lem}

\begin{proof}
Consider the pair of quadrangles given by \eqref{StandardPair}. Then $\Omega_1=1$, $\Omega_2=\sqrt{x_1x_2}$, and Equation \eqref{eq:CAMC}
can be written as
\begin{equation*}
\frac{a}{\sqrt{1+a}}\sqrt{x_1x_2}=\frac{b}{\sqrt{1+b}}
\end{equation*}
Observe that the compatibility equation $(1+a)(1+b)x_1=x_2$ holds. Thus
\begin{equation*}
ax_1=\bar{b}
\end{equation*}
and we obtain the solution
$$
b=\frac{x_1-x_2}{1-x_1},\ \ a=-\frac{x_1-x_2}{x_1(1-x_2)},
$$
thus proving the lemma.
\end{proof}

As a consequence, for a CAMC asymptotic net, the constant affine mean curvature $H$ is uniquely determined by the net. 

\subsection{Discrete affine minimal surfaces}

Given a pair of polygonal lines $\alpha:I\subset\mathbb{Z}\to\mathbb{R}^3$ and $\beta:J\subset\mathbb{Z}\to\mathbb{R}^3$, let $\nu:I\times J\to\mathbb{R}^3$ be defined by
$\nu(i,j)=\alpha(i)+\beta(j)$.
Observe that such a co-normal satisfies Moutard Equation with $\lambda=1$, and so
%$$
%\nu(i,j)+\nu(i+1,j+1)=\nu(i+1,j)+\nu(i,j+1).
%$$
defines an asymptotic net, which by Equation \eqref{eq:CAMC}, is CAMC with $H=0$. In other words, it is an affine minimal surface. In this way we can construct plenty of examples of affine minimal discrete asymptotic nets. If the polygonal lines $\alpha$ and $\beta$ are contained in a plane, then the asymptotic net is in fact a discrete improper affine sphere (\cite{Craizer2010},\cite{Kaferbock2013},\cite{Matsuura2003}).

%Conversely, any discrete affine minimal surface can be obtained by this construction (see \cite{Craizer2010} for details). 

We have the following characterization of affine minimal strips:

\begin{prop}\label{prop:MinimalParallelPlanes}
A discrete asymptotic net is affine minimal if and only if for each horizontal strip $(j+\tfrac{1}{2})$ the vectors $q_2(i,j+\tfrac{1}{2})$ are coplanar and, for each vertical strip $(i+\tfrac{1}{2})$ the vectors $q_1(i+\tfrac{1}{2},j)$ are coplanar. 
\end{prop}

\begin{proof}
Consider a pair of adjacent quadrangles, which we shall assume be given by \eqref{StandardPair}. By Lemma \ref{lemma:CAMCAdjacentQuadrangles}, the pair is CAMC with $H=0$ if and only if $x_1=x_2$, which is exactly the condition 
for the three vectors $D-B$, $C-A$ and $F-E$ being coplanar.
\end{proof}

\subsection{Discrete affine spheres}

According to \cite{Bobenko1999}, a pair $(q,\nu)$ is called a discrete (proper) affine sphere (with center at the origin) if it satisfies the discrete Lelieuvre's Equations \eqref{eq:Lelieuvre} and their duals
\begin{equation}\label{eq:LelieuvreDual}
\nu_1(i+\tfrac{1}{2},j)=-H q(i+1,j)\times q(i,j), \ \ \nu_2(i,j+\tfrac{1}{2})=H q(i,j+1)\times q(i,j),
\end{equation}
for any $(i,j)\in\mathcal{D}$. 
These dual relations say in particular that both nets $q$ and $\nu$ are asymptotic and
Moutard at the same time. Observe that we are choosing
$H$ instead of $1$, as it was done in \cite{Bobenko1999}, because we shall show that the discrete affine sphere is CAMC with affine mean curvature $H$. 

Consider a quadrangle $ABCD$, $A=q(i,j)$, $B=q(i+1,j)$, $C=q(i,j+1)$, $D=q(i+1,j+1)$. We shall denote
 $\nu(A)=\nu(i,j)$, $\nu(B)=\nu(i+1,j)$, $\nu(C)=\nu(i,j+1)$, $\nu(D)=\nu(i+1,j+1)$, $\Omega=\Omega(i+\tfrac{1}{2},j+\tfrac{1}{2})$, $\lambda=\lambda(i+\tfrac{1}{2},j+\tfrac{1}{2})$. 
Since $q$ is a Moutard net, there exists $\alpha>0$ such that 
\begin{equation}\label{eq:DualMoutard}
\alpha^2 (A+D)=B+C. 
\end{equation}
We shall call this equation dual Moutard. 

\begin{lem}
Denote $z=\nu_A\cdot A$. Then
$$
\nu(B)\cdot B=\nu(C)\cdot C=\nu(D)\cdot D=\nu(A)\cdot B=\nu(B)\cdot A=\nu(C)\cdot A=\nu(A)\cdot C=
$$
$$
=\nu(D)\cdot B=\nu(B)\cdot D=\nu(C)\cdot D=\nu(D)\cdot C=z.
$$
$$
\nu(D)\cdot A=\nu(A)\cdot D=z(2\lambda^{-2}-1)=z(2\alpha^{-2}-1), 
$$
$$ 
\nu(B)\cdot C=\nu(C)\cdot B=z(2\lambda^2-1)=z(2\alpha^2-1).
$$
As a consequence, $\alpha=\lambda$. 
\end{lem}
\begin{proof}
All formulas follow easily from Lelieuvre's \eqref{eq:Lelieuvre}, dual Lelieuvre \eqref{eq:LelieuvreDual}, Moutard \eqref{eq:Moutard} and dual Moutard \eqref{eq:DualMoutard} equations. 
\end{proof}

\begin{lem}
We have that
\begin{equation}\label{eq:Omegaq}
\Omega^2=2\frac{1-\lambda^2}{\lambda^2}\left[ A,B,C \right].
\end{equation}
\end{lem}

\begin{proof}
This formula follows from the definition of $\Omega$ \eqref{eq:DefineOmega} and the dual Moutard equation \eqref{eq:DualMoutard}. 
\end{proof}

\begin{prop}
We have that
$$
1-\lambda^2=\tfrac{H}{2}\lambda\Omega.
$$
\end{prop}

\begin{proof}
Observe that
$$
0=\nu(B)\cdot B-\nu(C)\cdot C=(\nu(B)-\nu(C)\cdot B+\nu(C)\cdot (B-C)=
$$
$$
=(\nu(B)-\nu(A)+\nu(A)-\nu(C))\cdot B+\nu(C)\cdot (B-A+A-C)
$$
$$
=H(A\times C)\cdot B+\nu(C)\cdot (\nu(A)\times \nu(B)).
$$
Thus, from Equations \eqref{eq:OmegaLambdaNu} and \eqref{eq:Omegaq}
$$
\frac{H\lambda^2}{2(1-\lambda^2)}\Omega^2=\lambda\Omega,
$$
which proves the proposition.
\end{proof}

From the above proposition we conclude that 
$$
-\tfrac{2a}{\sqrt{1+a}\Omega}=H
$$
and so a discrete proper affine sphere is a CAMC with affine mean curvature $H$. 

Consider the quadrangle $ABCD$ given by \eqref{StandardPair} and let 
$O=\tfrac{1}{2}(1,1,\tfrac{a+1}{a})$ be the center of the interpolating quadric. Then, since any quadric is an affine sphere, the following dual Lelieuvre's equations hold:
\begin{equation*}
\nu(B)-\nu(A)=-H(B-O)\times (A-O), \ \ \nu(C)-\nu(A)=H(C-O)\times (A-O),
\end{equation*}
and similarly for the edges $CD$ and $BD$.

\begin{lem}
The asymptotic net is a discrete affine sphere if and only if the centers
of a compatible interpolation are independent of the quadrangle. 
\end{lem}

\begin{proof}
Since dual Leliuvre's equations hold at each quadrangle, we have only to verify that the centers of all quadrangles 
coincide.
\end{proof}

\begin{lem}
Consider the adjacent pair of quadrangles given by \eqref{StandardPair}. The discrete asymptotic net is a discrete affine sphere if and only if it is CAMC and 
\begin{equation}\label{eq:AdjacentAS}
y_2-1=(1+b)y_1.
\end{equation}
\end{lem}

\begin{proof}
One can verify that the center $O_1$ of the quadrangle $EACF$ is 
$$
O_1=\frac{1}{2\bar{b}}\left( (\bar{b}+1)x_2-x_1, (\bar{b}+1)y_2-y_1-1,  (\bar{b}+1)x_2  \right).
$$
Since, by compatibility, $(\bar{b}+1)x_2=(1+a)x_1$, we obtain
$$
O_1=\frac{1}{2\bar{b}}\left( ax_1, (\bar{b}+1)y_2-y_1-1,  (a+1)x_1  \right).
$$
Recall that the CAMC is equivalent to $ax_1=\bar{b}$. Thus the first and third components of $O$ and $O_1$ coincide if and only if the CAMC condition holds. Since the second components of $O$ and $O_1$ coincide if and only if Equation \eqref{eq:AdjacentAS} holds, the lemma is proved. 
\end{proof}

\section{Some properties of CAMC asymptotic nets}

\subsection{Coincidence of generators}

Consider two adjacent compatible interpolating quadrics $\phi(u,v)$ and $\psi(s,t)$ that coincide at the line $i=i_0$. We would like to look for points of the common edge such that the generators of both quadrics at this point coincide.

Assume that the pair of quadrangles are given by \eqref{StandardPair} and the interpolating quadrics $\phi$ and $\psi$ are given by \eqref{StandardInterpolator} and \eqref{eq:Psi}. Then the generator of $\phi$ at $\phi(0,v)=(0,v,0)$ is the line 
$u\to\phi(u,v)$, while the generator of $\psi$ at $\psi(s,0)=(0,s,0)$ is given by $t\to\psi(s,t)$.  

\begin{lem}
Assume that the pair of quadrangles are given by \eqref{StandardPair} and the interpolating quadrics $\phi$ and $\psi$ are given by \eqref{StandardInterpolator} and \eqref{eq:Psi}. The condition for the coincidence of the generators at $v=s$ is
\begin{equation}\label{eq:CoincidenceGenerators}
(\bar{b}-ax_1)s^2+(ax_1+y_1+1-\bar{b}y_2-y_2)s-y_1=0
\end{equation}
\end{lem}

\begin{proof}
We are looking for points $(0,v)$ such that
$\phi_u(0,v)$ given by Equation \eqref{eq:PhiU} and $\psi_t(s,0)$ given by Equation \eqref{eq:PsiT}, $s=v$, are parallel. In homogeneous coordinates,
$$
\phi(0,v)=(0,v, 0, 1),\ \ \phi_u(0,v)=\left( 1+av, av, (1+a)v, av \right),
$$
and we can simplify $\psi_t(s,0)$ to obtain
$$
\psi_t(s,0)=\left( x_1(1+sa) , y_1 +s((1+\bar{b})y_2-y_1-1), sx_1(1+a), \bar{b}s \right).
$$
Consider the $3\times 4$ matrix formed by these three vectors. Since the first and third columns are linearly dependent, the generators coincide if and only if the determinant of the last three columns of these three vectors vanishes. Calculating this determinant we obtain condition \eqref{eq:CoincidenceGenerators}.
\end{proof}

\begin{prop}
The point $s=\infty$ is a solution of Equation \eqref{eq:CoincidenceGenerators} if and only if the asymptotic net is CAMC. In this case, Equation \eqref{eq:CoincidenceGenerators} reduces to
\begin{equation}\label{eq:CoincidenceGeneratorsCAMC}
\left( (\bar{b}+1)(1-y_ 2)+y_1 \right)s-y_1=0
\end{equation}
\end{prop}

\begin{proof}
Just recall that the condition for CAMC asymptotic net is $ax_1=\bar{b}$ (Lemma \ref{lemma:CAMCAdjacentQuadrangles}).
\end{proof}

In the following sections we shall see some interesting consequences of Equations \eqref{eq:CoincidenceGenerators} and \eqref{eq:CoincidenceGeneratorsCAMC}.

\subsection{Ruled CAMC asymptotic nets}

We say that a discrete asymptotic net is ruled if one of the two families of polygonal lines is in fact a family of straight lines. In other words, there are two possibilities: (1) For each fixed $i_0$, the polygonal line $j\to q(i_0,j)$ is a straight line, or (2)
for each fixed $j_0$, the polygonal line $i\to q(i,j_0)$ is a straight line. For the sake of definiteness, we shall assume in this section that (2) holds.

Consider a pair of adjacent quadrangles $(i-\tfrac{1}{2},j_0+\tfrac{1}{2})$ and $(i+\tfrac{1}{2},j_0+\tfrac{1}{2})$ with a compatible pair of interpolator quadrics $\phi(u,v)$ and $\psi(s,t)$. 
Since the asymptotic net is ruled, we have that the segments $u\to\phi(u,0)$ and $t\to\psi(0,t)$ are collinear, as well as the segments
$u\to\phi(u,1)$ and $t\to\psi(1,t)$. We would like to know if, for fixed $v_0$ and $t_0$ such that $\phi(0,v_0)=\psi(s_0,0)$, the segments $u\to \phi(u,v_0)$ and $t\to\psi(s_0,t)$ are collinear. In the positive case, we say that $\phi$ and $\psi$ have compatible rulings.

\begin{prop}
Consider a ruled strip of an asymptotic net. The following statements are equivalent:
\begin{enumerate}
\item
The strip is CAMC.

\item
There exist interpolating quadrics along the strip with compatible rulings.

\item
There exists one single quadric that interpolates all quadrangles of the strip. 
\end{enumerate}
\end{prop}

\begin{proof}
Consider a pair of adjacent quadrangles given by \eqref{StandardPair}. Since the strip is ruled, we have that $y_1=0$ and $y_2=1$. Let us begin with $(1)\to(2)$. Assuming the CAMC condition, Equation \eqref{eq:CoincidenceGeneratorsCAMC} becomes an identity, thus it holds for any $s$.  
 To prove that $(2)\to(3)$, observe first that
$$
\phi(u_0,0)=u_0(1,0,0), \ \ \phi(u_0,1)=\frac{1}{1+au_0}(u_0(1+a),1+au_0,u_0(1+a)).
$$
Write $x_1=u_0$. Since compatibility of rulings is equivalent
to the CAMC condition, we conclude that
$$
x_2=\frac{x_1(1+a)}{1+ax_1}=\frac{u_0(1+a)}{1+au_0}.
$$
So $\phi(u_0,1)=(x_2,1,x_2)$, thus proving (3). The implication $(3)\to(1)$ is trivial.
\end{proof}

\subsection{Discrete Demoulin Transforms of a CAMC asymptotic net}

In the smooth case, the intersection of the affine normal with the Lie quadric determines a surface that
is a Demoulin transform of the original surface if and only if the affine mean curvature is constant.
In this section, we describe a discrete version of this property.

The intersection of the normal line at $(u,v)=(0,0)$ with the quadric $\phi$ is 
\begin{equation}\label{eq:DiscreteDemoulinTransform}
Z=\frac{1}{a}\left( (1+a)D+A-B-C \right).
\end{equation}
corresponding to the parameters $u=v=\infty$.

Consider now the pair of quadrangles given by \eqref{StandardPair} with interpolating quadrics $\phi$ given by \eqref{StandardInterpolator} and $\psi$ given by \eqref{eq:Psi}. To calculate the tangent plane to the quadric $\phi$ at $Z$, we use homogeneous coordinates and consider
new variables $\bar{u}=\frac{1}{u}$, $\bar{v}=\frac{1}{v}$. We have
$$
\bar\phi(\bar{u},\bar{v})=\left[ \bar{u}\bar{v}+a: \bar{v}+a:\bar{u}+a:a+1 \right].
$$
At $\bar{u}=\bar{v}=0$,
\begin{equation*}\label{eq:TangentPlane}
\bar{\phi}=\left[a:a:a:a+1\right],\ \bar{\phi}_{\bar{u}}=\left[0:0:1:0\right],\ \bar{\phi}_{\bar{v}}=\left[0:1:0:0\right].
\end{equation*}
On the other hand, denoting by $Y$ given by \eqref{eq:DiscreteDemoulinTransform} in the quadrangle $ACEF$ we obtain
$$
[\bar{b}:(1+\bar{b})F-C-E]=\left[ \bar{b}: (1+\bar{b})x_2-x_1: (1+\bar{b})y_2-y_1-1:(1+\bar{b})x_2\right],
$$
Thus the point $Y$ belongs to the tangent plane to $\phi$ at $Z$ if and only if
\[
\left|
\begin{array}{cccc}
\bar{b} & (1+\bar{b})x_2-x_1 & (1+\bar{b})y_2-y_1-1 & (1+\bar{b})x_2\\
 0 & 1 & 0 & 0\\
0 & 0 & 1 & 0 \\
a & a & a & 1+a
\end{array}
\right|=0,
\]
which is equivalent to
$$
(1+a)\bar{b}=a(1+\bar{b})x_2.
$$
which is equivalent to the pair of quadrangles being CAMC.

%For projective results concerning the discrete Lie quadric, see \cite{Schief2017} and \cite{Schief2018}. 

\subsection{Discrete affine spheres and the Demoulin property}

In the case of an affine sphere, the surface is Demoulin (\cite{McCarthy-Schief},\cite{Schief-Szeres}). Using our techniques we can recover this result and prove also a converse in case of CAMC.

\begin{prop}
A discrete asymptotic net is an affine sphere if and only if it is CAMC and Demoulin.
\end{prop}

\begin{proof}
If the asymptotic net is an affine sphere, then Equations \eqref{eq:AdjacentAS} and \eqref{eq:CoincidenceGeneratorsCAMC}
imply that both solutions of the equation of coinciding generators are $s=\infty$. Thus we have a double root for Equation \eqref{eq:CoincidenceGenerators}. Since this fact holds for any pair of adjacent quadrangles, we conclude that the asymptotic net is discrete Demoulin (\cite{McCarthy-Schief},\cite{Schief-Szeres}). 

Conversely, for a CAMC discrete asymptotic net, if Equation \eqref{eq:CoincidenceGeneratorsCAMC} admits a double root then
Equation \eqref{eq:AdjacentAS} holds. This implies that the asymptotic net is an affine sphere. 
\end{proof}

\subsection{Discrete affine minimal asymptotic nets and $Q$-surfaces}

Consider a horizontal strip $(j+\tfrac{1}{2})$. By Proposition \ref{prop:MinimalParallelPlanes}, there exists a plane 
$\pi(j+\tfrac{1}{2})$ such that $q_2(i,j+\tfrac{1}{2})$ is parallel to $\pi$. The intersection of $\pi(j+\tfrac{1}{2})$ with the plane at infinity
is a line that we shall denote $l=l(j+\tfrac{1}{2})$.

\begin{lem}
The line $l(j+\tfrac{1}{2})$ is a common generator of the interpolating paraboloids of the quadrangles $(i,j+\tfrac{1}{2})$, for any $i$. 
\end{lem}

\begin{proof}
In the affine minimal case, we have that $x_1=x_2$ and the interpolating quadrics
$$
\phi(u,v)=(u:v:uv:1), \ \ \psi(s,t)=\left(tx_1: s+ty_1+st(y_2-y_1-1): stx_1:1\right),
$$
%$$
%\psi(s,t)=\left( s(0,1,0)+t(x_1,y_1,0)+st(x_2-x_1,y_2-y_1-1,x_2) :1 \right)
%$$
are paraboloids. The lines
$$
\phi(u,\infty)=(0:1:u:0), \ \ \psi(\infty,t)=(0: 1+t(y_2-y_1-1): tx_1:0)
$$
coincide, their equation being $X_1=X_4=0$, thus proving the lemma.
\end{proof}

\begin{cor}
A discrete affine minimal surface is a discrete $Q$-surface.
\end{cor}

\begin{proof}
By the above lemma, the common generators of the interpolating quadrics along a strip are collinear. Thus the discrete affine minimal surface is a discrete $Q$-surface (\cite[Def.7.3]{McCarthy-Schief}).
\end{proof}

\section*{\textbf{Acknowledgement}}

The authors are thankful to CAPES and CNPq for financial support during the preparation of this paper. They also thank Pontifical Catholic University of Rio de Janeiro.

% ----------------------------------------------------------------
\bibliographystyle{amsplain}

\end{document}